\documentclass[letterpaper,reqno]{amsart}

\usepackage{amsmath} 
\usepackage{amsfonts} 
\usepackage{amssymb} 
\usepackage{epsfig} 
\usepackage{graphicx} 
\usepackage{multirow} 
\usepackage{verbatim} 
\usepackage{rotating} 
\usepackage[all]{xy} 
\usepackage{appendix}

\newtheorem{theorem}{Theorem}[section]
\newtheorem{lemma}[theorem]{Lemma}

\theoremstyle{definition}
\newtheorem{definition}[theorem]{Definition}
\newtheorem{proposition}[theorem]{Proposition}

\theoremstyle{remark}
\newtheorem{remark}[theorem]{Remark}
\theoremstyle{notation}

\numberwithin{equation}{section}
\theoremstyle{corollary}

\usepackage[bookmarks=false]{hyperref}
\newcommand{\Map}{\mathrm{Map}}

\newcommand{\Top}{\mathsf{Top}}
\newcommand{\Chk}{\mathsf{dgMod}_{k}}

\newcommand{\C}{\mathsf{C}}

\newcommand{\Com}{\mathsf{Com}}
\newcommand{\Ass}{\mathsf{Ass}}
\newcommand{\PP}{\mathsf{P}}
\newcommand{\EE}{\mathsf{E}_{\infty}}
\newcommand{\QQ}{\mathsf{Q}}

\newcommand{\coh}{\mathrm{H}}

\newcommand{\Ho}{\mathrm{Ho}}
\newcommand{\Op}{\mathrm{Op}}

\newcommand{\Algk}{\mathsf{dgAlg}_{k}}
\newcommand{\EAlgk}{\mathsf{E_{\infty}dgAlg}_{k}}

\newcommand{\CAlgk}{\mathsf{dgCAlg}_{k}}

\begin{document}

\title[]{Comparing commutative and associative unbounded differential graded algebras over $\mathbb{Q}$ from homotopical point of view}

\author[]{Ilias Amrani}
\address{Department of Mathematics, Masaryk University\\ Kotlarska 2\\ Brno, Czech Republic.}
\email{ilias.amranifedotov@gmail.com}
\email{amrani@math.muni.cz}
\thanks{Supported by the project CZ.1.07/2.3.00/20.0003
of the Operational Programme Education for Competitiveness of the Ministry
of Education, Youth and Sports of the Czech Republic.
}

\thanks{}

\subjclass[2000]{Primary 55, Secondary 14 , 16, 18}




\begin{abstract}
In this paper we establish a faithfulness result, in a homotopical sense, between a subcategory of the model category of augmented differential graded commutative algebras CDGA and a subcategory of the model category of augmented differential graded algebras DGA over the field of rational numbers $\mathbb{Q}$. 
\end{abstract}

\maketitle
\section*{Introduction}
It is well known that the forgetful functor from the category of commutative $k$-algebras to the category of category of associative $k$-algebras is fully faithful. We have an analogue result between the category of unbounded differential graded commutative $k$-algebras $\CAlgk$ and the category of unbounded differential graded associative algebras $\Algk$. The question that we want explore is the following:
Suppose that $k=\mathbb{Q}$, is it true that forgetful functor $ U:\CAlgk\rightarrow \Algk$ induces a fully faithful functor at the level of homotopy categories
$$\mathbf{R}U: \Ho (\CAlgk) \rightarrow \Ho(\Algk).$$
The answer is \textbf{no}. A nice and easy counterexample was given by Lurie. He has considered $k[x,y]$ the free commutative CDGA in two variables concentrated in degree 0. It follows obviously  that 
$$\Ho (\CAlgk)(k[x,y], S)\simeq \coh^{0}(S)\oplus  \coh^{0}(S),$$
while 
$$\Ho (\Algk)(k[x,y], S)\simeq \coh^{0}(S)\oplus  \coh^{0}(S)\oplus \coh^{-1}(S).$$ 
Something nice happens if we consider the category of augmented CDGA denoted by $\CAlgk^{\ast}$ and augmented DGA denoted by  $\Algk^{\ast}$.
\begin{theorem}[\ref{thm}]
For any $R$ and $S$ in  $\CAlgk^{\ast}$, the induced map by the forgetful functor 
$$\Omega \Map_{\CAlgk^{\ast}}(R,S)\rightarrow \Omega  \Map_{\Algk^{\ast}}(R,S),$$
has a retract, in particular
$$\pi_{i}\Map_{\CAlgk^{\ast}}(R,S)\rightarrow \pi_{i}\Map_{\Algk^{\ast}}(R,S)$$
is injective $\forall~ i>0$.
\end{theorem}  
Let $S$ be a differential graded commutative algebra which is a "loop" of an other CDGA algebra $A$, i.e. $S= \mathsf{Holim}(k\rightarrow A\leftarrow k)$, 
where the homotopy limit is taken in the model category $\CAlgk$. A direct consequence of our theorem is that the right derived functor $\mathbf{R}U$ is a
faithful functor i.e., the induced map $~\Ho(\CAlgk^{\ast})(R,S)\rightarrow \Ho(\Algk^{\ast})(R,S)$ is injective. 
\subsection*{Interpretation of the result in the derived algebraic geometry}
Rationally, any pointed topological $X$ space can be viewed as an augmented (connective) commutative differential graded algebra via its cochain complex $C^{\ast}(X,\mathbb{Q})$. In case where $X$ is a  simply connected rational space, the cochain complex $C^{\ast}(X,\mathbb{Q})$ carries the whole homotopical information about $X$, by Sullivan Theorem \cite{hess2007rational}. Moreover, the bar construction $\mathrm{B}C^{\ast}(X,\mathbb{Q})$ is identified (as $\EE$-DGA) to $C^{\ast}(\Omega X,\mathbb{Q})$ and $\Omega C^{\ast}(X,\mathbb{Q})$ is identified (as $\EE$-DGA) to  $C^{\ast}(\Sigma X,\mathbb{Q})$ cf. \cite{fresse2010bar}. This interpretation allows us to make the following definition: A generalized rational pointed space is an augmented commutative differential graded $\mathbb{Q}$-algebra (possibly unbounded). In the same spirit, we define a pointed generalized $\textbf{noncommutative rational space}$ as  an augmented differential graded $\mathbb{Q}$-algebra  (possibly unbounded). Let $A$ be any augmented CDGA resp. DGA,  we will call a CDGA resp. DGA of the form $\Omega A$ a \textit{op-suspended} CDGA resp. DGA.
Our theorem \ref{thm}, can be interpreted as follows:\\
 \textbf{ The homotopy category of op-suspended generalized commutative rational spaces is a subcategory of the homotopy category of op-suspended generalized noncommutative rational spaces.}

\section{DGA, CDGA and $\mathsf{E}_{\infty}$-DGA.}
We work in the setting of unbounded differential graded $k$-modules $\Chk$. This is a a symmetric monoidal closed model category ($k$ is a commutative ring). 
We denote the category of (reduced) operads in $\Chk$ by $\mathsf{Op}_{k}$. We follow notations and definitions of \cite{berger2003axiomatic}, we say that an operad
$\mathsf{P}$ is \textit{admissible} if the category of $\mathsf{P}-\Algk$ admits a model structure where the fibrations are degree wise surjections and weak 
equivalence are quasi-isomorphisms. For any map of operads $\phi: \mathsf {P}\rightarrow \mathsf{Q}$ we have an induced adjunction of the corresponding 
categories of algebras:
$$\xymatrix{ \mathsf{P}-\Algk \ar@<2pt>[r]^{\phi_{!}} & \mathsf{Q}-\Algk. \ar@<2pt>[l]^{\phi^{\ast}} }$$
A $\Sigma$-cofibrant operad $\mathsf{P}$ is an operad such that $\mathsf{P}(n)$ is $k[\Sigma_{n}]$-cofibrant in $\mathsf{dgMod}_{k[\Sigma_{n}]}$. 
Any cofibrant operad $\mathsf{P}$ is a $\Sigma$-cofibrant operad \cite[Proposition 4.3]{berger2003axiomatic}. We denote the associative operad by $\mathsf{Ass}$ and the 
commutative operad by $\mathsf{Com}$.The operad $\mathsf{Ass}$ is an admissible operad and $\Sigma$-cofibrant, while the operad $\mathsf{Com}$ is not admissible in general.
In the rational case, when $k=\mathbb{Q}$ the operad $\mathsf{Com}$ is admissible but not $\Sigma$-cofibrant. More generally any cofibrant operad $\mathsf{P}$ is admissible \cite[Proposition 4.1, Remark 4.2]{berger2003axiomatic}. We define a symmetric tensor product of operads by the formulae
 $$[\PP\otimes \QQ](n)=\PP(n)\otimes \QQ(n),~~ \forall~ n\in\mathbb{N}.$$

\begin{lemma}\label{cofibration}
Suppose that $\phi:\mathsf{Ass}\rightarrow \PP$ is a cofibration of operads.The operad $\mathsf{P}$ is admissible and the functor $\phi^{\ast}:\mathsf{P}-\Algk\rightarrow \Algk$ preserves fibrations, weak equivalences and cofibrations with cofibrant domain in the inderleing category $\Chk$.
\end{lemma}
 
  \def\cocartesien{%
    \ar@{-}[]+L+<-6pt,+1pt>;[]+LU+<-6pt,+6pt>%
    \ar@{-}[]+U+<-1pt,+6pt>;[]+LU+<-6pt,+6pt>%
  }

\begin{proof}
First of all, the operad $\PP$ is admissible, indeed we use the cofibrant resolution $r:\EE\rightarrow \Com$ and consider the following pushout in $\Op_{k}$ given by:
$$\xymatrix{\Ass_{\infty}\ar@{^{(}->}[r]\ar@{>>}[d]^{\sim} & \EE\ar[d]^{\alpha}\\
\Ass\ar[r]^{f} & \EE^{'}}$$
Where $\Ass_{\infty}$ in the cofibrant replacement of $\Ass$ in $\Op_{k}$ and $\Ass_{\infty}\rightarrow \EE$ is a cofibration. Since the category $\Op_{k}$ is left proper in the sense of \cite[Theorem 3]{spitzweck2001operads}, we have that $\alpha:\EE\rightarrow \EE^{'}$ is an equivalence. We denote by $I$ the unit  interval in the category 
$\Chk$ which is strictly coassociative. The opposite endomorphism operad $\mathsf{End}^{op}(I)$ has a structure of $\EE$-algbra and $\Ass_{\infty}$-algebra which factors through $\Ass$ i.e., we have two compatible maps of operads:
$$\xymatrix{
\Ass_{\infty}\ar@{^{(}->}[r]\ar@{>>}[d]^{\sim} & \EE\ar[d]^{\alpha}\ar[rdd] &\\
\Ass\ar@{^{(}->}[r]^{f}\ar[rrd] & \EE^{'}\ar@{.>}[rd] \cocartesien&\\
& & \mathsf{End}^{op}(I)
}$$
by the universality of the pushout, we have a map of operads $\EE^{'}\rightarrow \mathsf{End}^{op}(I)$. This means that the unit  interval $I$ has a structure of $\EE^{'}$-colagebra \cite[p.4]{berger2003axiomatic}. 
Moreover, we have a commutative diagram in $\Op_{k}$ given by 
$$\xymatrix{
\Ass\ar[r]^-{\Delta}\ar@{^{(}->}[dd]^{\phi} & \Ass\otimes \Ass\ar[r]^{\phi\otimes f} & \PP\otimes \EE^{'}\ar@{>>}[dd]^{\sim}_{id\otimes r}\\
&&\\
\PP \ar[rr]^{id} & & \PP\otimes\Com=\PP
}$$
where the diagonal map $\Delta: \Ass\rightarrow \Ass\otimes \Ass$ is induced by the diagonals $\Sigma_{n}\rightarrow \Sigma_{n}\times \Sigma_{n}$. Hence, the map $\PP\otimes \EE^{'}\rightarrow \PP$ admits a section. It implies by \cite[Proposition 4.1]{berger2003axiomatic}, that $\PP$ is admissible and $\Sigma$-cofibrant.  
Since all objects in $\mathsf{P}-\Algk$ are fibrant and $\phi^{\ast}$ is a right Quillen adjoint, it preserves fibrations and weak equivalences. 

Since $\PP$ is an admissible operad, we have a Quillen adjunction 
$$\xymatrix{ \Algk \ar@<2pt>[r]^{\phi_{!}} & \PP-\Algk , \ar@<2pt>[l]^{\phi^{\ast}}  }$$
where the functor $\phi^{\ast}$ is identified to the forgetful functor. Moreover, the model structure on  $\PP-\Algk$ is the transfered model structure from the cofibrantly generated model structure $\Algk $ via the adjunction $\phi_{!},\phi^{\ast}$. Suppose that $f: A\rightarrow C$ is a cofibration in $\PP-\Algk$ such that $A$ is cofibrant in $\Chk$. We factor this map as a cofibration followed by a trivial fibration 
$$\xymatrix{A\ar@{^{(}->}[r]^{i} & P \ar@{->>}[r]^{p}_{\sim} & B}$$ 
in the category $\Algk$. By \cite[Lemma 4.1.16]{rezk1996spaces}, we have an induced map of endomorphism operads (of diagrmas): 
$$\mathrm{End}_{\{A\rightarrow P\rightarrow B\}}\rightarrow \mathrm{End}_{\{A\rightarrow B\}}$$
which is a trivial fibration. Moreover, we have the following commutative diagram in $\Op_{k}$
$$\xymatrix{
\Ass\ar@{^{(}->}[d]\ar[r] & \mathrm{End}_{\{A\rightarrow P\rightarrow B\}}\ar@{->>}[d]^{\sim}\\
\PP\ar[r] \ar@{.>}[ru] & \mathrm{End}_{\{A\rightarrow B\}}
} $$
Since $\Op_{k}$ is a model category, it implies that we have a lifting map of operads $\PP\rightarrow \mathrm{End}_{\{A\rightarrow P\rightarrow B\}}$, hence $i$ and $p$ are maps of $\PP-\Algk$. Therefore, we consider the following commutative square in the category $\PP-\Algk$
$$\xymatrix{
A \ar@{^{(}->}[d]_{f}\ar[r]^{i} & P\ar@{->>}[d]_{\sim}^{p}\\
B\ar[r]^{id}\ar@{.>}[ru]^{r} & B
} $$
the lifting map $r$ exists since $\PP-\Algk$ is a model category, we conclude that $p \circ r=id$ and $r\circ f=i$, which means that $f$ is a retract of $i$, hence $f$ is a cofibration in $\Algk$. 
\end{proof}   
\begin{remark}
With the same notation as in \ref{cofibration}, if $A$ is a cofibrant object in  $\PP-\Algk$ then $A$ is a cofibrant object in $\Chk$. Indeed $k\rightarrow A$ is a cofibration in $\PP-\Algk$, by the previous lemma $k\rightarrow A$ is a cofibration in $\Algk$. Therefore, $k\rightarrow A$ is a cofibration in $\Chk$. 
\end{remark}
\section{Suspension in CGDA and DGA}
We denote the the operad $\EE^{'}$ of the previous section by $\EE$, and $k=\mathbb{Q}$.
\subsection{$\mathsf{E}_{\infty}$-DGA}  
We have a map of operads  $\mathsf{Ass}\rightarrow\mathsf{Com}$, which we factor as cofibration followed by a trivial fibration. 
$$\xymatrix{\mathsf{Ass}\ar@{^{(}->}[r] & \mathsf{E}_{\infty}\ar@{>>}[r]^{\sim} & \mathsf{Com}
}$$
As a consequence, we have the following Quillen adjunctions
$$\xymatrix{ \Algk \ar@<2pt>[r]^{Ab_{\infty}} & \EAlgk \ar@<2pt>[l]^{U}  \ar@<2pt>[r]^{str} & \CAlgk \ar@<2pt>[l]^{U^{'}} }$$
These adjunctions have the following properties:
\begin{itemize}
\item The functors $U^{'}$ and $ U\circ U^{'}$ and are the forgetful functors, they are fully faithful cf  \ref{abelianization} and \ref{strictification}.
\item The functors $str, ~ U^{'}$ form a Quillen equivalence since $k=\mathbb{Q}$ cf \cite[Corollary 1.5]{kriz1995operads}. The functor $str$ is the strictification functor.
\item The functors $Ab_{\infty},~U$ form a Quillen pair. 
\item The composition $str \circ Ab_{\infty}$ is the abelianization functor $Ab:\Algk\rightarrow \CAlgk.$
\item The functors $str$ and $Ab$ are idenpotent functors. cf  \ref{abelianization} and \ref{strictification}.
\end{itemize}
The model categories $\CAlgk^{\ast}$ and $\Algk^{\ast}$ and $\EE\Algk^{\ast}$ are pointed model categories. It is natural to introduce the suspension functors in these categories. 

\begin{definition}\label{suspension}
Let $\C$ be any pointed model category, we denote the point by $1$, and let $A\in\C$, a suspension $\Sigma A$ is defined as $\mathsf{hocolim}(1\leftarrow A\rightarrow 1)$. 
\end{definition}

\begin{proposition}\label{strictification}
Any map $f:A\rightarrow S$ in $\EAlgk$, where $S$ is in $\CAlgk$ factors in a unique way as $A\rightarrow str(A)\rightarrow S$ and the forgetful functor $U^{'}: \CAlgk\rightarrow \EAlgk$ is fully faithful. Moreover, the unit of the adjunction $\nu_{A}:A\rightarrow str(A)$ is a fibration. 
\end{proposition}
\begin{proof}
Suppose that we have a map $h:R\rightarrow S$ in $\EAlgk$ such that $R$ and $S$ are objects in $\CAlgk$. 
By definition of the operad $\EE$ the map $h$ has to be associative, therefore $h$ is a morphism in $\CAlgk$ since $R$ and $S$ are commutative differential graded algebras. The forgetful functor $U^{'}: \CAlgk\rightarrow \EAlgk$ is fully faithful, this implies that $str(S)=S$ for any $S\in\CAlgk$. We have a commutative diagram induced by the unit $\nu$ of the adjunction $(U^{'}, ~str)$ :
$$\xymatrix{
A\ar[r]^{f}\ar[d]_{\nu_{A}} & S\ar[d]^-{\nu_{S}=id}\\
str(A)\ar[r]^{str(f)} & str(S)=S.
}$$
We conclude that $f=str(f)\circ\nu_{A}.$ The surjectivity of the $\nu_{A}$ follows from the universal property of 
$str(A)$. Hence, $\nu_{A}$ is a fibration in $\EAlgk.$ 
\end{proof}
\begin{proposition}\label{abelianization}
Any map $f:A\rightarrow S$ in $\Algk$, where $S$ is in $\CAlgk$ factors in a unique way as $A\rightarrow Ab(A)\rightarrow S$ and the forgetful functor $U\circ U^{'}: \CAlgk\rightarrow \Algk$ is fully faithful. Moreover, the unit of the adjunction $\nu_{A}:A\rightarrow Ab(A)$ is a fibration. 
\end{proposition}
\begin{proof}
The proof is the same as in \ref{strictification}.
\end{proof}
\begin{proposition}\label{unite}
Suppose that we have a trivial cofibration $k\rightarrow \underline{k}$ in $\EAlgk$. Then the universal map $\pi: ~Ab(\underline{k})\rightarrow str(\underline{k})$ is a trivial fibration and admits a section in the category $\CAlgk$.
\end{proposition}
\begin{proof}
We consider the following commutative diagram in $\EAlgk$: 
$$\xymatrix{
k\ar[r]^{\sim}\ar[d]_{id} & \underline{k}\ar[d]^-{}\\
k=str(k)\ar[r]^{\sim} & str(\underline{k}).
}$$
The map $k\rightarrow str(\underline{k})$ is an equivalence since $str$ is left Quillen functor, the same thing holds for the abelianization functor i.e., $\underline{k}\rightarrow Ab(\underline{k})$ is a trivial fibration, since $k\rightarrow \underline{k}$ is a trivial cofibration in $\Algk$ \ref{cofibration} and $Ab$ is a left Quillen functor. On another hand the map $\underline{k}\rightarrow str(\underline{k})$, which is a trivial fibration in $\EAlgk$ and hence in $\Algk$, can be factored (cf \ref{abelianization}) as  $\underline{k}\rightarrow Ab(\underline{k})\rightarrow str(\underline{k})$, where $Ab(\underline{k})\rightarrow str(\underline{k})$ is a trivial fibration between cofibrant object in $\CAlgk$. It follows that we have a retract $l: str(\underline{k})\rightarrow Ab(\underline{k})$.
\end{proof}

\begin{definition}
The suspension functor in the pointed model categories $\CAlgk^{\ast}$, $\Algk^{\ast}$ and $\EAlgk^{\ast}$  are denoted by $\mathrm{B}$, $\Sigma$ and $\mathrm{B}_{\infty}$. 
\end{definition}

\begin{lemma}\label{suspension}
Suppose that $A$ is a cofibrant object in $\EAlgk^{\ast}$, and $i:A\rightarrow \underline{k}$ a cofibration, then $str(\mathrm{B}_{\infty}A)$ is a retract of $Ab(\Sigma A)$ in the category $\CAlgk$.
\end{lemma}
\begin{proof}
First of all if a map $f$ is associative, commutative resp. $\EE$-map we put an index $f_{a},~f_{c}$ resp. $f_{\infty}$, notice that by definition of the operad $\EE$ any $\EE$-map is a strictly associative map.  Suppose that $A$ is a cofibrant object in $\EAlgk$. Consider the following commutative square:
$$\xymatrix{
A\ar@{^{(}->}[rr]^{i_{\infty}} \ar@{^{(}->}[dd]^{i_{\infty}} & & \underline{k}\ar@{^{(}->}[dd]_{h_{a}} \ar@{.>}[rddd]^{f_{\infty}}  & \\
&&&\\
\underline{k}\ar@{^{(}->}[rr]^{h_{a}}  \ar@{.>}[rrrd]_{f_{\infty}} & &\Sigma A \ar@{.>}[rd]^-{\exists ! }_-{u_{a}}\cocartesien &  \\
& & &  \mathrm{B}_{\infty}A \cocartesien,
}$$
where $\Sigma A$ is the (homotopy \ref{cofibration}) pushout in $\Algk$ and $\mathrm{B}_{\infty}A$ is the (homotopy) pushout in $\EAlgk$. By proposition \ref{strictification} and proposition \ref{abelianization} we have a following commutative square in $\Algk$: 
$$\xymatrix{
\Sigma{A}\ar[rr]^{u_{a}}\ar[dd] & & \mathrm{B}_{\infty}A\ar[dd]\\
&& \\
Ab(\Sigma A)\ar[rr]^-{x_{c}} && str [\mathrm{B}_{\infty}A]= \mathrm{B}[str (A)].
}$$  

By \ref{unite} we have an inclusion of commutative differential graded algebras $l_{c}: str(\underline{k})\rightarrow Ab(\underline{k})$ and after strictification we obtain on another (homotopy) pushout square in $\CAlgk$ given by 
$$\xymatrix{
str(A)\ar@{^{(}->}[rr]^{i_{c}} \ar@{^{(}->}[dd]^{i_{c}} & & str(\underline{k})\ar@{^{(}->}[dd]_{f_{c}} \ar@{^{(}->}[r]^{l_{c}}  & Ab(\underline{k})\ar@{^{(}->}[ddd] ^{h_{c}}\\
&&&\\
str(\underline{k})\ar@{^{(}->}[rr]^{f_{c}}  \ar@{^{(}->}[d]_{l_{c}}  & &\mathrm{B} [str(A)] \ar@{.>}[rd]^{\exists ! }_-{u_{c}} \cocartesien &  \\
Ab(\underline{k})\ar@{^{(}->}[rrr]^{h_{c}} & & &  Ab(\Sigma(A)).
}$$
In order to prove that $\mathrm{B}[str (A)]$ is a retract of $Ab(\Sigma(A))$ it is sufficient to prove that 
$$ x_{c}\circ h_{c}\circ l_{c}=f_{c}.$$
By proposition \ref{strictification} and proposition \ref{abelianization}, the composition  $\EE$-maps 
$$\xymatrix{ \underline{k}\ar[r]^{f_{\infty}} & \mathrm{B}_{\infty}A\ar[r] &  str [\mathrm{B}_{\infty}A]
}$$
can be factored in a unique way as 
$$\xymatrix{ \underline{k}\ar[r] & Ab(\underline{k})\ar[r]^{\pi} & str(\underline{k})\ar[r]^-{\alpha_{c}} & str [\mathrm{B}_{\infty}A]=\mathrm{B}[str (A)].
}$$
By unicity, $\alpha_{c}=f_{c}$. On another hand, using the first poushout in $\EAlgk$, the previous composition $\underline{k}\rightarrow str [\mathrm{B}_{\infty}A]$ is factored  as 
$$\xymatrix{ \underline{k}\ar[r]^{h_{a}} &\Sigma A \ar[r]  & Ab(\Sigma A)\ar[r]^{x_{c}} &  str[\mathrm{B}_{\infty}A].
}$$
We summarize the previous remarks in the following commutative diagram:
$$\xymatrix{
\underline{k}\ar[r]^{pr}\ar[d]^{id} & Ab(\underline{k})\ar[r]^{\pi}\ar@{.>}[d]^{h_{c}} & str(\underline{k})\ar[r]^{f_{c}} &  str[\mathrm{B}_{\infty}A]\ar[d]^{id}\\
\underline{k} \ar[r] & Ab(\Sigma A) \ar[rr]^{x_{c}} & &  str[\mathrm{B}_{\infty}A]
}$$
by definition of $h_{a}$, the doted map $h_{c}$ makes the left square commutative. Since the whole square is commutative and the map $pr$ is surjective we conclude that $x_{c}\circ h_{c}=f_{c}\circ\pi$. Since the map $l_{c}: Str(\underline{k})\rightarrow Ab(\underline{k})$ is a retract of $\pi$ (Cf. \ref{unite}) i.e., $\pi\circ l_{c}= id$, we conclude that $x_{c}\circ h_{c}\circ l_{c}=f_{c}$. Finally, by unicity of the pushout, we deduce that the following composition 
$$\xymatrix{ \mathrm{B}[str (A)] \ar[r]^-{u_{c}} & Ab(\Sigma A) \ar[r]^{x_{c}}  & \mathrm{B}[str (A)]
}$$
is identity. 
\end{proof}
\section{Main result and applications}
\begin{theorem}\label{thm}
For any $R$ and $S$ in  $\CAlgk^{\ast}$, the induced map by the forgetful functor  
$$\Omega \Map_{\CAlgk^{\ast}}(R,S)\rightarrow \Omega  \Map_{\Algk^{\ast}}(R,S),$$
has a retract, in particular 
$$\pi_{i}\Map_{\CAlgk^{\ast}}(R,S)\rightarrow \pi_{i}\Map_{\Algk^{\ast}}(R,S)$$
is injective $\forall~ i>0$. 
\end{theorem}
\begin{proof}
Suppose that $R$ is (cofibrant) object in $\EAlgk$ and $S$ any object in $\CAlgk$. By adjunction, we have
that 
\begin{eqnarray}
 \Omega\Map_{\CAlgk^{\ast}}(str(R), S) & \sim & \Map_{\CAlgk^{\ast}}(\mathrm{B}[str(R)], S)\\
 & \sim &\Map_{\CAlgk^{\ast}}(str[\mathrm{B}_{\infty}R],S)\\
 &\sim& \Map_{\EAlgk^{\ast}}(\mathrm{B}_{\infty}R,S)\\
 &\sim & \Omega\Map_{\EAlgk^{\ast}}(R,S). 
\end{eqnarray}
By Lemma \ref{suspension}, we have a retract 
$$\Map_{\CAlgk^{\ast}}(\mathrm{B}[str(R)], S)\rightarrow \Map_{\CAlgk^{\ast}}(Ab(\Sigma R), S) \rightarrow \Map_{\CAlgk^{\ast}}(\mathrm{B}[str(R)], S).$$
Again by adjunction:
$$\Map_{\CAlgk^{\ast}}(Ab(\Sigma R), S)\sim \Map_{\Algk^{\ast}}(\Sigma R,S)\sim \Omega\Map_{\Algk^{\ast}}(R,S).$$
We conclude that
$$\xymatrix{ \Omega \Map_{\EAlgk^{\ast}}(R,S)\ar[r]^{U} & \Omega\Map_{\Algk^{\ast}}(R,S)\ar[r] &\Omega \Map_{\EAlgk^{\ast}}(R,S)}$$
is a retract. Hence, the forgetful functor $U$ induces a injective map on homotopy groups i.e.,
$$\pi_{i}\Map_{\CAlgk^{\ast}}(str(R),S)\simeq\pi_{i}\Map_{\EAlgk^{\ast}}(R,S)\rightarrow \pi_{i}\Map_{\Algk^{\ast}}(R,S)$$
is injective $\forall~ i>0$. 
\end{proof}
\subsection{Rational homotopy theory }
We give an application of our theorem \ref{thm} in the context of rational homotopy theory. Let $X$ be a simply connected rational space such that $\pi_{i}X$ is finite dimensional $\mathbb{Q}$-vector space for each $i>0$. Let $C^{\ast}(X)$ be the differential graded $\mathbb{Q}$-algebra cochain associated to $X$ which is a connective $\EAlgk$.  
By Sullivan theorem $\pi_{i}X\simeq \pi_{i}\Map_{\CAlgk^{\ast}}(C^{\ast}(X), \mathbb{Q}).$ 
By \ref{thm}, we have that $\pi_{i}X$ is a sub $\mathbb{Q}$-vector space of $\pi_{i}\Map_{\Algk^{\ast}}(R,S)$.
On another hand \cite{amrani2013mapping}, since $C^{\ast}(X)$ is  connective, we have that for any $i>1$ 
$$\pi_{i}\Map_{\Algk^{\ast}}(C^{\ast}(X),\mathbb{Q})\simeq\mathrm{HH}^{-1+i}(C^{\ast}(X),\mathbb{Q}),$$
where $\mathrm{HH}^{\ast}$ is the Hochschild cohomology. Since we have assumed finiteness condition on $X$, we have that  
$$\mathrm{HH}^{-1+i}(C^{\ast}(X),\mathbb{Q})\simeq\mathrm{HH}_{i-1}(C^{\ast}(X),\mathbb{Q}).$$
The functor $C^{\ast}(-, \mathbb{Q}): \Top^{op}\rightarrow \EAlgk$ commutes with finite homotopy limits, where $\Top$ is the category of simply connected spaces. Hence, 
$$\mathrm{HH}_{-1+i}(C^{\ast}(X),\mathbb{Q})=\mathrm{H}^{i-1}[C^{\ast}(X)\otimes^{\mathbf{L}}_{C^{\ast}(X\times X)}\mathbb{Q}]  \simeq \mathrm{H}^{i-1}(\Omega X,\mathbb{Q}).$$
We conclude that $\pi_{i}X$  is a sub $\mathbb{Q}$-vector space of $\mathrm{H}^{i-1}(\Omega X,\mathbb{Q}).$\\ 
More generally by Block-Lazarev result \cite{block2005andre} on rational homotopy theory and \cite{amrani2013mapping}, we have an injective map of $\mathbb{Q}$-vector spaces
$$\mathrm{AQ}^{-i}(C^{\ast}(X),C^{\ast}(Y))\rightarrow \mathrm{HH}^{-i+1}(C^{\ast}(X),C^{\ast}(Y)),$$
where the $C^{\ast}(X)$-(bi)modules structure on $C^{\ast}(Y)$ is given by $C^{\ast}(X)\rightarrow \mathbb{Q}\rightarrow C^{\ast}(Y)$, and $\mathrm{AQ}^{\ast}$ is the Andr\'e-Quillen cohomology. We also assume  that $X$ and $Y$ are simply connected and $i>1$. \\
More generally,  
$$\pi_{i}\Map_{\EAlgk}(R,S)=\mathrm{AQ}^{-i}(R,S)\rightarrow \mathrm{HH}^{-i+1}(R,S)=\pi_{i}\Map_{\Algk}(R,S)$$
is an injective map of  $\mathbb{Q}$-vector spaces for all $i>1$ and any augmented $\EE$-differential graded connective $\mathbb{Q}$-algebras $R$ and $S$, where the action of $S$ on $R$ is given by $S\rightarrow \mathbb{Q}\rightarrow R$.\\

\textbf{Acknowledgement :} I'm grateful to Beno\^it Fresse for his nice explanation of Lemma \ref{cofibration}, 
the key point of the proof is due to him.

\bibliographystyle{plain} 
\bibliography{rational}

\begin{thebibliography}{1}

\bibitem{amrani2013mapping}
Ilias Amrani.
\newblock The mapping space of unbounded differential graded algebras.
\newblock {\em arXiv preprint arXiv:1303.6895}, 2013.

\bibitem{berger2003axiomatic}
C.~Berger and I.~Moerdijk.
\newblock Axiomatic homotopy theory for operads.
\newblock {\em Commentarii Mathematici Helvetici}, 78(4):805--831, 2003.

\bibitem{block2005andre}
Jonathan Block and Andrej Lazarev.
\newblock Andr{\'e}--{Q}uillen cohomology and rational homotopy of function
  spaces.
\newblock {\em Advances in Mathematics}, 193(1):18--39, 2005.

\bibitem{fresse2010bar}
Benoit Fresse.
\newblock The bar complex of an {E}-infinity algebra.
\newblock {\em Advances in Mathematics}, 223(6):2049--2096, 2010.

\bibitem{hess2007rational}
Kathryn Hess.
\newblock Rational homotopy theory: a brief introduction.
\newblock {\em Contemporary Mathematics}, 436:175, 2007.

\bibitem{kriz1995operads}
Igor Kriz and J~Peter May.
\newblock {\em Operads, algebras, modules and motives}.
\newblock Soci{\'e}t{\'e} math{\'e}matique de France, 1995.

\bibitem{rezk1996spaces}
Charles~Waldo Rezk.
\newblock {\em Spaces of algebra structures and cohomology of operads}.
\newblock PhD thesis, Massachusetts Institute of Technology, 1996.

\bibitem{spitzweck2001operads}
Markus Spitzweck.
\newblock Operads, algebras and modules in general model categories.
\newblock {\em arXiv preprint math/0101102}, 2001.

\end{thebibliography}

\end{document}